\documentclass[12pt, reqno]{amsart}
\usepackage{amsmath,amssymb,amsthm}
 2

\usepackage[colorlinks=true,linkcolor=blue,citecolor=blue,pdfpagelabels=false]{hyperref}
\newtheorem{thm}{Theorem}[section]
\newtheorem{lem}[thm]{Lemma}

\newtheorem{defn}[thm]{Definition}
\newcommand{\thmref}[1]{Theorem~\ref{#1}}
\newcommand{\lemref}[1]{Lemma~\ref{#1}}

\theoremstyle{remark}
\newtheorem{rmk}{Remark}[section]

\newenvironment{acknowledgements}{\bigskip\textbf{Acknowledgements.}}{}

\renewcommand{\geq}{\geqslant}

\newcommand{\medmatrix}[4]{\Bigl(\begin{matrix} #1 \!& \!#2 \\[-4pt] #3 \!&\! #4 \end{matrix}\Bigr)}

\title[Construction of cusp forms using Rankin-Cohen Brackets]{ Construction of cusp forms using Rankin-Cohen Brackets}
\author{Abhash Kumar Jha and Arvind Kumar}
\address{School of Mathematical Sciences, National Institute of Science Education and Research, Bhubaneswar 751005, India}
\email{abhash.jha@niser.ac.in}

\address {Harish-Chandra Research Institute, Chhatnag Road, Jhunsi,Allahabad -211 019, India.}
\email{kumararvind@hri.res.in}

\subjclass[2010]{Primary: 11F37; Secondary: 11F25, 11F66}
\keywords{Modular Forms, Rankin-Cohen Brackets, Adjoint map, Dirichlet-Series}

\begin{document}
\begin{abstract} 
For a fix modular form $g$ and a non negative ineteger $\nu$, by using Rankin-Cohen bracket we first
define a linear map $ T_{g, \nu}$ on the space of modular forms.
We explicitly compute the adjoint of this map and show that
%We compute the Petersson scalar product of a cusp form $f$ with Rankin-Cohen bracket of a  modular form 
%$g$ and $n$-th Poinca{\'r}e series and show that this product is 
the $n$-th Fourier coefficients of the 
image of the cusp form $f$ under this map is, upto a constant a special value of  Rankin-Selberg convolution of $f$ 
and $g.$ This is a generalization of the work due to W. Kohnen (Math. Z., 207,  (1991), 657-660) and  
S. D. Herrero (Ramanujan J., 36(2014), no.3, 529-536) in the case of integral weight modular forms to
half integral weight modular forms. As a consequence we get non-vanishing of special value of certain 
Rankin- Selberg convolution of modular forms.
\end{abstract}

\date{\today}

\maketitle

\section{Introduction}
W. Kohnen \cite{kohnen} constructed cusp forms whose Fourier coefficients are given by special 
values of certain Dirichlet series by computing the adjoint of the product map by a fixed 
cusp form with respect to  the usual Petersson scalar product. This result has been generalized 
by several authors to other automorphic forms (see the list \cite{ckk, lee, lee2,sakata, wang}).
The work of Kohnen has been generalized by S. D. Herrero \cite{herrero}, where the 
author constructed the cusp forms by computing the adjoint of the map constructed using 
the  Rankin-Cohen brackets by a fixed cusp form instead of the product map. Recently, the work 
of S. D. Herrero \cite{herrero} has been generalised by first author and B. Sahu to the case of Jacobi
forms \cite{a-s} which also generalises the result of H. Sakata  \cite{sakata}. In this article
we extend the work of S. D. Herrero to the case of  half integral weight modular forms. We 
apply this result to get non-vanishing of special value of certain Rankin- Selberg convolution 
of modular forms.

\section{Preliminaries}
\subsection{Elliptic Modular Forms} 
Let $\mathcal H$ be the complex upper half-plane and $\Gamma $ be a congruence subgroup of 
the full modular group $SL_2(\mathbb{Z}).$ For $k\in\mathbb{Z}$ and
$\gamma = \medmatrix a b c d ,$ define the slash operator as follows;\\
  $$f\mid_k \gamma (z):= (cz+d)^{-k}f(\gamma z),\;{\rm{where}}\; \gamma z = \dfrac{az+b}{cz+d}.$$

Let $M_k(\Gamma, \chi)$ (respectively $S_k(\Gamma, \chi)$) denote the space of modular forms (resp. cusp forms) of integral weight $k$ and 
character $\chi$ for $\Gamma,$ i.e., for every 
$\gamma = \medmatrix a b c d \in \Gamma,\;f\mid_k \gamma (z)\;=\;\chi (d)f(z),$ 
and holomorphic at cusps of $\Gamma.$\\

We define the Petersson scalar product on $S_k(\Gamma, \chi)$ as follows;
$$
\langle f, g \rangle\;=\;\int_{\Gamma\setminus \mathbb{H}} f(z)\overline{g(z)}  (Im(z))^k d^*z,
$$
where $ z=x+iy $ and $d^*z=\dfrac{dxdy}{y^2}$ is an invariant measure under the 
action on $\Gamma$ on $\mathcal H .$ For more details on the theory of modular 
forms, we refer to \cite{kob}.

\subsection{Poinca{\'r}e series}
\begin{defn}
Let $ n$ be a be a positive integer. The $n$-th Poinca{\'r}e series of integer weight $k$ is defined by
\begin{equation}\label{poincare}
P_{k, n} (z):=\sum_{\gamma \in \Gamma{\infty}\setminus \Gamma} e^{2 \pi i nz}|_k \gamma ,
\end{equation}
where $\Gamma_{\infty}:=\left\{ \pm \medmatrix 1 t 0 1 | t  \in \mathbb Z  \right\}.$  It is well known that $P_{k, n} \in S_k(\Gamma)$ for $k >2$ .
\end{defn} 
\noindent This series has the following property. 
\begin{lem}\label{poincare-lemma}
Let $ f \in S_k(\Gamma)$ with Fourier expansion $f(z)=\sum_{m=1}^\infty a(m) q^m .$ Then 
\begin{equation}
\langle f,~  P_{k, n} \rangle = \alpha_{k, n} a(n), \;\;\;\;\;\;\; where \;\; \alpha_{k, n}=\frac{\Gamma(k-1)} {(4\pi n)^{k-1}}.
\end{equation}

\end{lem}
\subsection{Modular Forms of Half Integral Weight}
Let $\Gamma =\Gamma_0(4).$ For $k\in\mathbb{Z}$ and $\gamma = \medmatrix a b c d ,$ 
define the slash operator as follows;\\
$$f\tilde{\mid}_{k+\frac{1}{2}} \gamma (z):= \left(\dfrac{c}{d}\right)\left(\dfrac{-4}{d}\right)^{k+\frac{1}{2}} (cz+d)^{-k-\frac{1}{2}}f(\gamma z),$$ where $\left(\dfrac{c}{d}\right)$ 
is the Kronecker symbol.\\
Let $M_{k+\frac{1}{2}}(\Gamma, \chi)$ (resp. $S_{k+\frac{1}{2}}(\Gamma, \chi)$) denote the space of modular forms (resp. cusp forms) of weight $k+\frac{1}{2}$ and 
character $\chi$ for $\Gamma,$ i.e., for every $\gamma = \medmatrix a b c d \in \Gamma,\;f\tilde{\mid}_{k+\frac{1}{2}} \gamma (z)\;=\;\chi(d)f(z),$ 
and holomorphic (resp. vanish) at cusps of $\Gamma.$ \\

We define the Petersson scalar product on $S_{k+\frac{1}{2}}(\Gamma, \chi)$ as follows;
$$
\langle f, g \rangle\;=\;\int_{\Gamma\setminus \mathbb{H}} f(z)\overline{g(z)}  (Im(z))^{k+\frac{1}{2}} d^*z,
$$
where $z=x+iy.$
 The spaces $S_k(\Gamma,\chi)$ and $S_{k+\frac{1}{2}}(\Gamma,\chi)$ are finite dimensional Hilbert spaces.\\
 For more details on the theory of modular forms of half integral weight, we refer to \cite{kob} and \cite{shi}.\\
\subsection{Poinca{\'r}e series of half integral weight}
\begin{defn}
Let $ n$ be a be a positive integer. The $n$-th Poinca{\'r}e series of weight $k+\frac{1}{2},$ where $k\in \mathbb{Z}$ is defined by
\begin{equation}\label{poincare half}
P_{k+\frac{1}{2}, n} (z):=\sum_{\gamma \in \Gamma{\infty}\setminus \Gamma} e^{2 \pi i nz}\tilde{\mid}_{k+\frac{1}{2}} \gamma.
\end{equation}
It is well known that $P_{k+\frac{1}{2}, n} \in S_{k+\frac{1}{2}}(\Gamma)$ for $k >2$ .
\end{defn} 
\noindent This series has the following property. 

\begin{lem}\label{poincare-lemma half}
Let $ f \in S_{k+\frac{1}{2}}(\Gamma)$ with Fourier expansion 
$f(z)=\sum_{m=1}^\infty a(m) q^m .$ Then 
\begin{equation}
\langle f,~  P_{k+\frac{1}{2}, n} \rangle= \tilde{\alpha}_{k, n} a(n), \;\;\;\;\;\; 
where \;\; \tilde{\alpha}_{k, n}=\frac{\Gamma(k-\frac{1}{2})} {(4\pi n)^{k-\frac{1}{2}}}.
\end{equation}

\end{lem}

 The following lemmas tell about the growth of the Fourier coefficients of a modular form. 
\begin{lem}\label{convergence}\cite{iwa}
If $f\in M_k (\Gamma,\chi)$ with Fourier coefficients $a(n),$ then 
$$
a(n) \ll |n|^{k-1+\epsilon}, 
$$  
and moreover, if $f$ is a cusp form, then
$$
a(n) \ll |n|^{\frac{k}{2}-\frac{1}{4}+\epsilon}.
$$  
\end{lem}
\begin{lem}\label{convergence half}
If $f\in M_{k+\frac{1}{2}} (\Gamma,\chi)$ with Fourier coefficients $a(n),$ then 
$$
a(n) \ll |n|^{k-\frac{1}{2}+\epsilon}, 
$$  
and moreover, if $f\in S_{k+\frac{1}{2}} (\Gamma,\chi)$ is a cusp form, then
$$
a(n) \ll |n|^{\frac{k}{2}+\epsilon}.
$$  
\end{lem}
\subsection{Rankin-Cohen Brackets} 
Let $k$ and $l$ be real numbers and  $\nu \ge 0$ be an integer.  Let $f$ and $g$ be two complex valued holomorphic functions on $\mathcal H. $
 Define the $\nu$-th Rankin-Cohen bracket of $f$ and $g$ by 
\begin{equation}
[f,g]_\nu :=\sum_{r=0}^\nu C_r(k,l; \nu)D^rf D^{\nu-r}g,
\end{equation}
 where $D^rf=\dfrac{1}{(2\pi i)^r}\dfrac{d^rf}{dz^r}$ and $C_r(k,l; \nu)~=~(-1)^{\nu-r}\binom{\nu}{r}\dfrac{\Gamma(k+\nu)\Gamma (l+\nu)}{\Gamma(k+r)\Gamma (l+\nu-r)}$ and $\Gamma (x)$ is the usual Gamma function.
\begin{rmk}
It is easy to verify that
\begin{equation}\label{lm}
[f|_k \gamma, g|_l \gamma ]_\nu= [f, g]|_{k+l+2\nu} \gamma, ~~ \forall \gamma \in \Gamma.
\end{equation}
\end{rmk}
\begin{rmk}
We note that the $0$-th Rankin-Cohen bracket is the usual product of modular forms i.e.,  $ [f, g]_0= fg.$ 
\end{rmk}
\begin{thm}\cite{co}
Let $\nu \ge 0$ be an integer and $f \in M_k(\Gamma,\chi_1)$ and $g \in M_l(\Gamma,\chi_2).$ 
Then $ [f, g]_\nu \in M_{k+l+2\nu}(\Gamma, \chi_1\chi_2\chi)$,\\ \\
\hspace{.2in} where $\chi =\begin{cases}
                                                         1,&\;if\;both\;k,l\in\mathbb{Z}, \\ \chi_{-4}^k,&\;if\;k\in \mathbb{Z}\;and\;l\in \mathbb{Z}+\frac{1}{2},\\ \chi_{-4}^l,&\;if\;k\in \mathbb{Z}+\frac{1}{2}\; and\;l\in \mathbb{Z},\\ \chi=\chi_{-4}^{k+l}&\;if\; both \;k,l\in \mathbb{Z}+\frac{1}{2},
                                                       \end{cases}$\\ \\
Moreover if  $\nu >0,$ then $ [f, g]_\nu \in S_{k+l+2\nu}(\Gamma,\chi_1\chi_2\chi).$
 In fact, 
$[~, ~]_\nu$ is a bilinear map from $M_k(\Gamma,\chi_1) \times M_l(\Gamma,\chi_2)$ to $M_{k+l+2\nu}(\Gamma, \chi_1\chi_2\chi).$
Here $\chi_{-4}$ is the character defined by $\chi_{-4}(x)= (\frac{-4}{x}).$
\end{thm}
Let $k,l \in \frac{\mathbb{Z}}{2} $ and $\nu \ge 0$ be integers and $\Gamma$ be a congruence subgroup of the full modular 
group $SL_2(\mathbb{Z}).$ Also assume that that $\Gamma \subseteq \Gamma_0(4)$ if either of $k$ or $l$ is non integer.
For a fixed $g \in M_l(\Gamma,\chi_2),$ we define the map 
$$T_{g, \nu}: S_k (\Gamma) \rightarrow S_{k+l+2\nu} (\Gamma, \chi_2)$$
defined by  $T_{g, \nu} (f)= [f,g]_\nu.$ $T_{g, \nu}$ is a $\mathbb C$-linear map of finite dimensional Hilbert spaces and therefore has an adjoint map $T_{g, \nu}^*: S_{k+l+2\nu} (\Gamma, \chi_2) \rightarrow S_k (\Gamma)$ such that
$$
\langle f, T_{g, \nu}(h) \rangle = \langle  T_{g, \nu}^*(f), h \rangle, ~~~ \forall f \in S_{k+l+2\nu} (\Gamma, \chi_2)~~~ {\rm and} ~~~h \in S_k (\Gamma).
$$
In \cite{herrero} S.D. Herrero computed the adjoint map for the case when $k, l \in \mathbb{Z}, \Gamma=SL_2(\mathbb{Z})$ 
and $\chi_2$ is the trivial character.
\begin{thm}\label{herrero}\cite{herrero} Let $k\geq 6$ and $l$ be natural numbers and $\nu\geq 0$. 
Let $g \in M_l(SL_2(\mathbb{Z}))$ with Fourier expansion 
$$
g(z)= \sum_{m=0}^\infty b(m) q^m.
$$
Suppose that either (a) $g$ is a cusp form or (b) $g$ is not cusp form and $l<k-3.$
Then the image of any cusp form $f \in S_{k+l+2\nu}(SL_2(\mathbb{Z}))$ with Fourier expansion 
$$
f(z)= \sum_{m=1}^\infty a(m) q^m
$$
under $T_{g, \nu}^*$ is given by 
$$
T_{g, \nu}^*(f)(z)=  \sum_{n=1}^\infty c(n) q^{n},
$$
where
\begin{equation}\label{eq2}
c(n)= \beta (k,l,\nu;n)L_{f,g,\nu,n}(\gamma),
\end{equation} where $L_{f,g,\nu,n}$ is the $L$- function associated with $f$ and $g,$ defined by\\
for $s\in \mathbb{C},$
$$L_{f,g,\nu,n}(s)=\sum_{m=1}^\infty\frac{a(n+m)\overline{b(m)}~~\alpha(k,l,\nu,n,m)}{(n+m)^{s}}$$
with 
$$\alpha (k,l,\nu,n,m)\;=\;\sum_{r=0}^\nu(-1)^{\nu-r}\binom{\nu}{r}\frac{\Gamma(k+\nu)\Gamma (l+\nu)}{\Gamma(k+r)\Gamma (l+\nu-r)}n^rm^{\nu-r}$$
and
$$\gamma=k+l+2\nu-1,\;\beta (k,l,\nu;n)=\frac{\Gamma (k+l+2\nu-1)~n^{k-1}}{\Gamma (k-1)(4\pi)^{l+2\nu}}.$$
\end{thm}
\begin{rmk}
One can prove the similar result for the case when $\Gamma$ is a congruence subgroup of level $N$ and $\chi_2$ is any
character mod $N$ using the technique used in proof of \thmref{main2}.  
\end{rmk}
\section{Statement of the Theorem}
Consider the following maps:
\begin{enumerate}
\item \label{case1}
$T_{g, \nu}: S_{k+\frac{1}{2}} (\Gamma) \rightarrow S_{k+l+2\nu+1} (\Gamma, \chi_2\chi),\;{\rm{with}}\;g\in M_{l+\frac{1}{2}}(\Gamma,\chi_2),$
\item \label{case2}
$T_{g, \nu}: S_k (\Gamma) \rightarrow S_{k+l+2\nu+\frac{1}{2}} (\Gamma, \chi_2\chi),\;{\rm{with}}\;g\in M_{l+\frac{1}{2}}(\Gamma,\chi_2),$
\item \label{case3}
$T_{g, \nu}: S_{k+\frac{1}{2}} (\Gamma) \rightarrow S_{k+l+2\nu+\frac{1}{2}} (\Gamma, \chi_2\chi),\;{\rm{with}}\;g\in M_l(\Gamma,\chi_2)$
\end{enumerate}
We exhibit explicitly the Fourier coefficients of $T_{g, \nu}^* (f)$ for $f \in S_{k+l+2\nu+1} (\Gamma, \chi_2\chi)$ in (\ref{case1}) 
and by using the same method, we can find the analogous maps in (\ref{case2}) and (\ref{case3}) (see the remark \ref{rmkcase23}).
These involve special values of certain 
Dirichlet series of Rankin- Selberg type associated to $f$ and $g.$ We now state the main theorem.
\begin{thm}\label{main2} Let $k$ and $l$ be natural numbers and $\nu\geq 0$. 
Let $g \in M_{l+\frac{1}{2}}(\Gamma,\chi_2)$ with Fourier expansion 
$$
g(z)= \sum_{m=0}^\infty b(m) q^m.
$$
Suppose that either (a) $g$ is a cusp form and $k>2$ or (b) $g$ is not cusp form and $l<k-\frac{3}{2}.$
Then the image of any cusp form $f \in S_{k+l+2\nu+1}(\Gamma, \chi_2\chi)$ with Fourier expansion 
$$
f(z)= \sum_{m=1}^\infty a(m) q^m
$$
under $T_{g, \nu}^*$ is given by 
$$
T_{g, \nu}^*(f)(z)=  \sum_{n=1}^\infty c(n) q^{n},
$$
where
\begin{equation}\label{eq2}
c(n)= \beta (k,l,\nu;n)L_{f,g,\nu,n}(\gamma),
\end{equation}
where
$$\gamma=k+l+2\nu,\;\beta (k,l,\nu;n)=\frac{\Gamma (k+l+2\nu)~n^{k-\frac{1}{2}}}{\Gamma (k-\frac{1}{2})(4\pi)^{l+2\nu+\frac{1}{2}}}$$
and $L_{f,g,\nu,n}(\gamma)$ is defined in \thmref{herrero}.
\end{thm}
\begin{rmk}\label{rmkcase23}
We have the similar results for the map in (\ref{case2}) with
$$\gamma=k+l+2\nu-\frac{1}{2},\;{\rm{and}}\;\beta (k,l,\nu;n)=\frac{\Gamma (k+l+2\nu-\frac{1}{2})~n^{k-1}}{\Gamma (k-1)~(4\pi)^{l+2\nu+\frac{1}{2}}},$$ 
and for the map in (\ref{case3}) with
$$\gamma=k+l+2\nu-\frac{1}{2},\;{\rm{and}}\;\beta (k,l,\nu;n)=\frac{\Gamma (k+l+2\nu-\frac{1}{2})~n^{k-\frac{1}{2}}}{\Gamma (k-\frac{1}{2})~(4\pi)^{l+2\nu}},$$
with the assumption that either (a) $g$ is a cusp form and $k>3$ or (b) $g$ is not cusp form and $l<k-2.$ 
\end{rmk}
\begin{rmk}
 Using \lemref{convergence} and \lemref{convergence half} one can show that the  series appearing in \eqref{eq2} converges.
\end{rmk}
\section{Proof of \thmref{main2}}
We need the following lemma to proof the main theorem.
\begin{lem}\label{mainlemma}
Using the same notation in \thmref{main2}, we have 
\begin{equation*}
  \sum_{\gamma\in\Gamma_{\infty}\setminus \Gamma} 
 \int_{\Gamma\setminus\mathcal{H}}  
\mid f(z)~\overline{[e^{2\pi inz}\mid_k\gamma, g ]_\nu}~(Im(z))^{k+l+2\nu+1}\mid ~d^*z
\end{equation*}
converges. 
\end{lem}
\begin{proof}
The proof is similar to Lemma 1 in \cite{herrero}.
\end{proof}
Now we give a proof of \thmref{main2}. Put 
$$
T_{g, \nu}^*(f)(z)=  \sum_{n=1}^\infty c(n) q^{n}. 
$$
Now, we consider the $n$-th Poinca{\'r}e series of weight $k+\frac{1}{2}$ as given in \eqref{poincare half}. Then using the \lemref{poincare-lemma half}, we have 
$$
\langle T_{g,\nu}^*f,  P_{k+\frac{1}{2},n}\rangle= \tilde{\alpha}_{k,n}c(n),
$$
where 
$$
\tilde{\alpha}_{k,n} = \frac{\Gamma(k-\frac{1}{2})}{(4 \pi n)^{k-\frac{1}{2}}}.
$$
On the other hand, by definition of the adjoint map we have
$$
\langle T_{g,\nu}^*f,  P_{k+\frac{1}{2},n}\rangle= \langle f,   T_{g,\nu}(P_{k+\frac{1}{2},n})\rangle= \langle f,   [P_{k+\frac{1}{2},n}, g]_\nu  \rangle.
$$
Hence we get 
\begin{equation}\label{fc}
c(n)= \frac{(4 \pi n)^{k-\frac{1}{2}}}{\Gamma(k-\frac{1}{2})}\langle  f,[P_{k+\frac{1}{2},n}, g]_\nu\rangle.
\end{equation}
By definition, 
\begin{eqnarray*}
\displaystyle \langle  f,[P_{k+\frac{1}{2},n}, g]_\nu\rangle &=& \int_{\Gamma\setminus\mathcal{H}}f(z)~\overline{\left[P_{k+\frac{1}{2},n}, g \right]_\nu (z)}~(Im(z))^{k+l+2\nu+1}~d^*z\\
&=&  \int_{\Gamma\setminus\mathcal{H}}f(z)~\overline{[\sum_{\gamma\in\Gamma_\infty\setminus \Gamma}e^{2\pi inz} \tilde{\mid}_{k+\frac{1}{2}}\gamma, g ]_\nu (z)}~(Im(z))^{k+l+2\nu+~1}~d^*z\\
%&=& \int_{\Gamma\setminus\mathcal{H}}f(z)~\sum_{\gamma\in\Gamma_\infty\setminus \Gamma}\overline{[e^{2\pi inz}\tilde{\mid}_{k+\frac{1}{2}}\gamma, g ]_\nu (z)}~(Im(z))^{k+l+2\nu+1}~d^*z\\
&=& \int_{\Gamma\setminus\mathcal{H}}\sum_{\gamma\in\Gamma_\infty\setminus \Gamma}f(z)~\overline{[e^{2\pi inz}\tilde{\mid}_{k+\frac{1}{2}}\gamma, g ]_\nu (z)}~(Im(z))^{k+l+2\nu+1}~d^*z.
\end{eqnarray*}
By \lemref{mainlemma}, we can interchange the sum and integration in $\langle  f,[P_{k,n}, g]_\nu\rangle$. Hence we get, 
\begin{equation*}
 \langle  f,[P_{k+\frac{1}{2},n}, g]_\nu\rangle\; =\; \sum_{\gamma\in\Gamma_\infty\setminus \Gamma} \int_{\Gamma\setminus\mathcal{H}}f(z)~\overline{[e^{2\pi inz}\tilde{\mid}_{k+\frac{1}{2}}\gamma, g ]_\nu (z)}~(Im(z))^{k+l+2\nu+1}~d^*z.
 \end{equation*}
 Since $g\in M_{l+\frac{1}{2}}(\Gamma, \chi_2),~~g\tilde{\mid}_{l+\frac{1}{2}}\gamma~=~\chi_2(d)g(z),$ for every $\gamma = \medmatrix a b c d \in \Gamma.$
 Therefore
 \begin{eqnarray*}
 \langle  f,[P_{k+\frac{1}{2},n}, g]_\nu\rangle\;&=&\; \!\!\!\!\!\!\!\!\!\!\!\!\!\!\!\!\sum_{\gamma=\medmatrix a b c d\in\Gamma_\infty\setminus \Gamma} \!\!\!\!\!\!\!\!\!\!\!\!\!\!\int_{\Gamma\setminus\mathcal{H}}f(z)~\overline{[e^{2\pi inz}\tilde{\mid}_{k+\frac{1}{2}}\gamma, \frac{1}{\chi_2(d)}~g\tilde{\mid}_{l+\frac{1}{2}}\gamma ]_\nu (z)}~(Im(z))^{k+l+2\nu+1}~d^*z\\
 &=& \!\!\!\!\!\!\!\!\!\!\!\!\!\!\!\!\sum_{\gamma =\medmatrix a b c d\in\Gamma_\infty\setminus \Gamma}\!\!\!\!\!\!\!\!\!\!\!\!\!\!\!\!\overline{\left(\frac{(\frac{-4}{d})^{k+l+1}}{\chi_2(d)}\right)} \int_{\Gamma\setminus\mathcal{H}}\!\!\!\!\!\!\!\!f(z)~\overline{[e^{2\pi inz}\mid_{k+\frac{1}{2}}\gamma, g\mid_{l+\frac{1}{2}}\gamma ]_\nu (z)} (Im(z))^{k+l+2\nu+1}~d^*z. 
\end{eqnarray*}
Using the change of variable $z$ to $\gamma^{-1}z$ in each integral, $\langle  f,[P_{k+\frac{1}{2},n}, g]_\nu\rangle$ equals
\begin{eqnarray*}
\!\!\sum_{\gamma =\medmatrix a b c d\in\Gamma_\infty\setminus \Gamma}\!\!\!\!\!\!\!\!\!\!\!\!\!\!\!\overline{\left(\frac{(\frac{-4}{d})^{k+l+1}}{\chi_2(d)}\right)} \!\!\!\int_{\Gamma\setminus\mathcal{H}}\!\!\!\!\!\!\!f(\gamma^{-1}z)~\overline{[e^{2\pi inz}\mid_{k+\frac{1}{2}}\gamma, g\mid_{l+\frac{1}{2}}\gamma ]_\nu (\gamma^{-1}z)}~(Im(\gamma^{-1}z))^{k+l+2\nu+1}~d^*(\gamma^{-1}z). 
\end{eqnarray*}
Since $f\in S_{k+l+2\nu+1}(\Gamma, \chi_2\chi),f(\gamma^{-1}z)=\chi_2(d)\chi(d)(cz+d)^{k+l+2\nu+1}f(z),$ for every $\gamma = \medmatrix a b c d \in \Gamma,$ and hence
\begin{eqnarray*}
 \langle  f,[P_{k+\frac{1}{2},n}, g]_\nu\rangle\;=\; \!\!\!\!\!\!\!\!\sum_{\gamma =\medmatrix a b c d\in\Gamma_\infty\setminus \Gamma}\overline{\left(\frac{(\frac{-4}{d})^{k+l+1}}{\chi_2(d)}\right)} \int_{\Gamma\setminus\mathcal{H}}\chi_2(a)\chi(a)(-cz+a)^{k+l+2\nu+1}f(z)\\
 \times~ \overline{(-cz+a)^{k+l+2\nu+1}([e^{2\pi inz}\mid_{k+\frac{1}{2}}\gamma, g\mid_{l+\frac{1}{2}}\gamma ]_\nu\mid_{k+l+2\nu+1}\gamma^{-1}) (z)} ~\left(\frac{Im(z)}{|-cz+a|^2}\right)^{k+l+2\nu+1}~d^*z. 
\end{eqnarray*}
Now using \eqref{lm}, we get
\begin{eqnarray*}
 \langle  f,[P_{k+\frac{1}{2},n}, g]_\nu\rangle\; &=&\; \!\!\!\!\!\!\!\!\!\!\!\!\!\!\!\!\!\!\!\!\sum_{\gamma=\medmatrix a b c d \in\Gamma_\infty\setminus \Gamma}\!\!\!\!\!\!\!\!\!\!\overline{\left(\frac{(\frac{-4}{d})^{k+l+1}}{\chi_2(d)}\right)} \chi_2(a)\chi(a) \int_{\gamma\Gamma\setminus\mathcal{H}}f(z)~\overline{[e^{2\pi inz}, g ]_\nu}~(Im(z))^{k+l+2\nu+1}~d^*z\\
&=& \!\!\!\!\!\!\!\!\!\!\!\!\!\!\!\!\!\!\!\!\sum_{\gamma=\medmatrix a b c d \in\Gamma_\infty\setminus \Gamma}\!\!\!\!\!\!\!\!\!\!\!\!\!\overline{\left(\frac{(\frac{-4}{d})^{k+l+1}}{\chi_2(d)}\right)} \chi_2(a)\left(\frac{-4}{a}\right)^{k+l+1} \!\!\!\!\!\!\int_{\gamma\Gamma\setminus\mathcal{H}}\!\!\!\!\!\!\!\!\!\!f(z)~\overline{[e^{2\pi inz}, g ]_\nu}~(Im(z))^{k+l+2\nu+1}~d^*z.
\end{eqnarray*}
The quantity appearing before integral is equals to 1, for all $\medmatrix a b c d \in\Gamma_\infty\setminus \Gamma,$ hence we get\\
$$\langle  f,[P_{k+\frac{1}{2},n}, g]_\nu\rangle\; =\;\sum_{\gamma\in\Gamma_\infty\setminus \Gamma}\int_{\gamma\Gamma\setminus\mathcal{H}}f(z)~\overline{[e^{2\pi inz}, g ]_\nu}~(Im(z))^{k+l+2\nu+1}~d^*z. $$
Now using Rankin unfolding argument, we have
\begin{eqnarray}\label{integration}
 \langle  f,[P_{k+\frac{1}{2},n}, g]_\nu\rangle\; &=&\;\int_{\Gamma_\infty \setminus\mathcal{H}}f(z)~\overline{[e^{2\pi inz}, g ]_\nu}~(Im(z))^{k+l+2\nu+1}~d^*z\\ \nonumber
&=&\int_{\Gamma_\infty \setminus\mathcal{H}} f(z)\sum_{r=0}^\nu C_r(k,l; \nu)~\overline{D^r(e^{2\pi inz}) D^{\nu-r}(g)}~(Im(z))^{k+l+2\nu+1}~d^*z\\ \nonumber
%&=&\sum_{r=0}^\nu C_r(k,l; \nu)\int_{\Gamma_\infty \setminus\mathcal{H}}f(z)~\overline{D^r(e^{2\pi inz}) D^{\nu-r}(g)}~(Im(z))^{k+l+2\nu+1}~d^*z. 
\end{eqnarray}
Now replacing $f$ and $g$ by their Fourier series in \eqref{integration}, $\langle  f,[P_{k+\frac{1}{2},n}, g]_\nu\rangle$ equals 
\begin{eqnarray*}
 \sum_{r=0}^\nu C_r(k,l; \nu)\int_{\Gamma_\infty \setminus\mathcal{H}} \left(\sum_s a(s)e^{2\pi is z}\right)n^r~\overline{e^{2\pi inz}}~m^{\nu-r}~\overline{b(m)}~\overline{e^{2\pi imz}}~(Im(z))^{k+l+2\nu+1}~d^*z
\end{eqnarray*}
\begin{eqnarray*} 
  &=& \int_{\Gamma_\infty \setminus\mathcal{H}} \sum_s\sum_m \alpha (k,l,\nu,n,m)a(s)\overline{b(m)}~e^{2\pi is z}~\overline{e^{2\pi inz}}~\overline{e^{2\pi imz}}~ (Im(z))^{k+l+2\nu+1}~d^*z\\
&=&  \sum_s\sum_m \alpha (k,l,\nu,n,m)a(s)\overline{b(m)}\int_{\Gamma_\infty \setminus\mathcal{H}}e^{2\pi is z}~\overline{e^{2\pi inz}}~\overline{e^{2\pi imz}}~(Im(z))^{k+l+2\nu+1}~d^*z.
\end{eqnarray*} 
%Putting $z=x+iy$, $\langle  f,[P_{k+\frac{1}{2},n}, g]_\nu \rangle$ equals
%\begin{eqnarray}\label{integration2}
%\nonumber &&  \sum_s\sum_m \alpha (k,l,\nu,n,m)a(s)\overline{b(m)}\int_{\Gamma_\infty \setminus\mathcal{H}} e^{2\pi i (s-n-m)x}e^{-2\pi (\alpha +n+m)y}y^{k+l+2\nu+1}\frac{dx dy}{y^2}.
%\end{eqnarray} 
 A fundamental domain for the action of $\Gamma_{\infty}$ on $\mathbb{H}$ is given by $[0,1]\times [0, \infty).$ 
 Integrating on this region after substituting $z=x+iy$,
 
 \begin{eqnarray*}
 \langle f,[P_{k,n}, g]_\nu \rangle 
 &=& \sum_s\sum_m \alpha (k,l,\nu,n,m)a(s)\overline{b(m)}\int_0^1\int_0^\infty  e^{2\pi i (s -n-m)x}e^{-2\pi (\alpha +n+m)y}y^{k+l+2\nu-1}dx dy\\
&=& \sum_m \alpha (k,l,\nu,n,m)a(n+m)\overline{b(m)}\int_0^\infty e^{-4\pi (n+m)y}y^{k+l+2\nu-1} dy\\
&=&  \frac{\Gamma(k+l+2\nu)} {(4\pi)^{k+l+2\nu}} \sum_m \frac{a(n+m)\overline{b(m)}\alpha (k,l,\nu,n,m)}{(n+m)^{k+l+2\nu}}.\\
\end{eqnarray*}

Now substituting the above value of  $\langle f,[P_{k+\frac{1}{2},n}, g]_\nu \rangle$  in \eqref{fc}, we get the required expression for $c(n)$ given in \thmref{main2}.\\
\section{Applications}
Consider the linear map $T_{g,\nu}^*\circ T_{g,\nu}$ on $S_k(\Gamma)$ with $g(z)\in M_l(\Gamma, \chi_2).$ If $\lambda$ is a eigenvalue of $T_{g,\nu}^*\circ T_{g,\nu},$ then $\lambda\geq 0.$ Suppose that $S_k(\Gamma)$ is one dimensional space generated by $f(z)=\sum_m a(n)q^n.$ Then $T_{g,\nu}^*\circ T_{g,\nu}(h)\;=\;\lambda f,\;\forall\;h\in S_k(\Gamma).$ In particular, $T_{g,\nu}^*\circ T_{g,\nu}(f)\;=\;\lambda f$ with $\lambda\geq 0$ and if we write $T_{g,\nu}^*\circ T_{g,\nu}(f)=\sum_nc(n)q^n$ then 
$$c(n)= \frac{\Gamma(k+l+2\nu-1)}{\Gamma(k-1)}\frac{n^{k-\frac{1}{2}}}{(4\pi)^{l+2\nu}}  \sum_{m=1}^\infty\frac{a_{T_{g,\nu}(f)}(n+m)\overline{b(m)}~~\alpha(k,l,\nu,n,m)}{(n+m)^{k+l+2\nu-1}},$$ 
where $a_{T_{g,\nu}(f)}(n)$ is the $n$-th Fourier coefficient of $T_{g,\nu}(f)=[f,g]_\nu.$ If $a(m_0)$ is the first non-zero Fourier coefficient of $f$ then by comparing the Fourier coefficients in  $T_{g,\nu}^*\circ T_{g,\nu}(f)\;=\;\lambda f,$ we have 
%$$c(m_0)= \frac{\Gamma(k+l+2\nu-1)}{\Gamma(k-1)}\frac{m_0^{k-\frac{1}{2}}}{(4\pi)^{l+2\nu}}  \sum_{m=1}^\infty\frac{a_{T_{g,\nu}(f)}(m_0+m)\overline{b(m)}~~\lambda(k,l,\nu,m_0,m)}{(m_0+m)^{k+l+2\nu-1}}=\lambda ~a(m_0),$$ 
$$\lambda=\frac{\Gamma(k+l+2\nu-1)}{a(m_0)\Gamma(k-1)}\frac{m_0^{k-\frac{1}{2}}}{(4\pi)^{l+2\nu}}  \sum_{m=1}^\infty\frac{a_{T_{g,\nu}(f)}(m_0+m)\overline{b(m)}~~\alpha(k,l,\nu,m_0,m)}{(m_0+m)^{k+l+2\nu-1}} \geq 0.$$
In particular, if we take $l=0,\;k=6$ and $\nu=0$ with $g(z)=\theta (z)=\sum\limits_nq^{n^2}$ and the unique newform $\Delta_{4,6}(z)=\sum_n\tau_{4,6}(n)q^n\in S_6(\Gamma_0(4)),$ in case (\ref{case2}) then\\
$m_0=1,\;\alpha (k,l,\nu,m_0,m)=1,$ and 
$$\lambda =\frac{\Gamma (\frac{11}{2})}{\Gamma(5)2\sqrt{\pi}} \sum_{m=1}^\infty\frac{a_{T_{\theta,0}(\Delta_{4,6})}(m+1)\overline{b(m)}}{(m+1)^{\frac{11}{2}}} > 0,$$
or
\begin{equation}\label{eqrmk}
\sum_{m=1}^\infty\frac{a_{T_{\theta,0}(\Delta_{4,6})}(m+1)\overline{b(m)}}{(m+1)^{\frac{11}{2}}} > 0.
\end{equation}
Now $a_{T_{\theta,0}(\Delta_{4,6})}(m+1)$ is the $(m+1)$-th Fourier coefficient of $\theta (z) \Delta_{4,6}(z)$ and equals to $\sum\limits_{r=1}^{m+1}b(r)\tau_{4,6}(m+1-r).$ 
Putting the value of $a_{T_{\theta,0}(\Delta_{4,6})}(m+1)$ in \eqref{eqrmk}, we have 
\begin{equation*}
\sum_{m=1}^\infty\frac{\left(\sum\limits_{r=1}^{m+1}b(r)\tau_{4,6}(m+1-r)\right)\overline{b(m)}}{(m+1)^{\frac{11}{2}}} > 0,
\end{equation*}
or
\begin{eqnarray*}
\sum_{m=1}^\infty\frac{\left(\sum\limits_{r=1}^{m^2+1}\tau_{4,6}(m^2+1-r^2)\right)}{(m^2+1)^{\frac{11}{2}}} > 0.
\end{eqnarray*}

\begin{acknowledgements}
Authors would like to thank B. Ramakrishnan and B. Sahu for useful discussions. The  authors would also like to thank Council of Scientific and Industrial Research
{\bf(CSIR)}, India for  financial support.  
\end{acknowledgements}

\end{document}